\documentclass[12pt,twoside,reqno]{amsart}
\usepackage[OT1]{fontenc}
\usepackage{type1cm}
\usepackage{amsmath}
\usepackage{amssymb}
\usepackage{wasysym}
\usepackage[dvips]{graphicx} 
\usepackage{psfrag}
\usepackage{verbatim}        

\numberwithin{equation}{section}

\theoremstyle{plain}
\newtheorem{theorem}{Theorem}[section]
\newtheorem{corollary}[theorem]{Corollary}
\newtheorem{proposition}[theorem]{Proposition}
\newtheorem{lemma}[theorem]{Lemma}

\theoremstyle{remark}
\newtheorem{remark}[theorem]{Remark}

\newtheorem{question}[theorem]{Question}

\theoremstyle{definition}

\newcommand{\R}{\mathbb{R}}
\newcommand{\Rn}{\R^n}

\newcommand{\N}{\mathbb{N}}
\newcommand{\PP}{\mathcal{P}}
\newcommand{\eps}{\varepsilon}
\newcommand{\roo}{\varrho}
\newcommand{\HH}{\mathcal{H}}

\newcommand{\yli}[2]{\genfrac{}{}{0pt}{1}{#1}{#2}}

\newcommand{\ud}[2]{\overline{D}_{h}(#1,#2)}        
\newcommand{\ld}[2]{\underline{D}_{h}(#1,#2)}       
\newcommand{\uds}[2]{\overline{D}_{h_s}(#1,#2)}       
\newcommand{\lds}[2]{\underline{D}_{h_s}(#1,#2)}      
\DeclareMathOperator{\dimp}{dim_p}
\DeclareMathOperator{\dist}{dist}
\DeclareMathOperator{\diam}{diam}
\DeclareMathOperator{\por}{por}

\begin{document}

\title[Conical upper density theorems and porosity]%
{Conical upper density theorems and porosity of measures}

\author{Antti K\"aenm\"aki}
\author{Ville Suomala} 

\address{Department of Mathematics and Statistics \\
         P.O. Box 35 (MaD) \\
         FI-40014 University of Jyv\"askyl\"a \\
         Finland}

\email{antakae@maths.jyu.fi}
\email{visuomal@maths.jyu.fi}

\thanks{AK acknowledges the support of the Academy of Finland (project
  \#114821).}
\subjclass[2000]{Primary 28A75; Secondary 28A78, 28A15.}
\keywords{Conical upper density, porosity, finite lower density,
  packing measure.}
\date{\today}

\begin{abstract}
  We study how measures with finite lower density are
  distributed around $(n-m)$-planes in small balls in $\Rn$. We also
  discuss relations between conical upper density theorems and
  porosity. Our results may be
  applied to a large collection of Hausdorff and packing type measures.
\end{abstract}

\maketitle

\section{Introduction}

Conical density theorems are used in geometric measure theory to
derive geometric information from given metric
information. Classically, they deal with the distribution of the
$s$-dimensional Hausdorff measure, $\HH^s$. 
The main applications of conical density theorems concern
rectifiability, see \cite{ma}, but they have been applied
also elsewhere in geometric measure theory, for example, in
the study of porous sets, see \cite{ma2} and \cite{KS}.
The upper conical density results, going back to
Besicovitch \cite{Bes} and Marstrand \cite{Mar}, show that under
certain conditions
there is a lot of $A$ near each $k$-dimensional linear subspace of
$\R^n$ in some small balls $B(x,r)$.
Besides Besicovitch and
Marstrand, the theory of upper conical density theorems has been
developed by Morse and Randolph 
\cite{MR}, Federer \cite{Fe}, and Salli \cite{Sa}. For a partial
survey on various conical density theorems 
for measures on $\Rn$, consult \cite{vk}. A sample result is
the following (Salli \cite[Theorem 3.1]{Sa}): If
$V\in G(n,n-m)$, where $G(n,n-m)$ denotes the space of all
$(n-m)$-dimensional linear subspaces of $\Rn$, $0<\alpha<§1$,
$A\subset\R^n$, $0<\HH^s(A)<\infty$, and $s>m\ge 1$, then
\begin{equation}\label{thm:salli}
\limsup_{r\downarrow 0}\frac{\HH^s\bigl(A\cap
  X(x,r,V,\alpha)\bigr)}{(2r)^s}\geq c  
\end{equation}
for $\HH^s$-almost all $x\in A$, where $c>0$ is a constant depending
only on $n,m,s$, and $\alpha$. Here
\begin{equation*}
  X(x,V,r,\alpha) = \{ y \in B(x,r)\,:\, \dist(y-x,V) <
  \alpha|y-x|\},
\end{equation*}
where $B(x,r)\subset\Rn$ is the closed ball with center at $x$ and
radius $r>0$. Open balls are denoted by $U(x,r)$.
Clearly, \eqref{thm:salli} is not true anymore if $s\leq
m$ since in this case it might happen that $A\subset V^\perp$.

In \cite{ma2}, Mattila improved the above result
by showing that it is not necessary to fix $V$ in
\eqref{thm:salli}. More precisely, he proved that if $A\subset\R^n$,
$0<\HH^s(A)<\infty$, $s>m$, and $0<\alpha<1$, then for a constant
$c>0$ depending only on $n$, $m$, $s$, and $\alpha$, 
\begin{equation}\label{thm:mattila}
 \limsup\limits_{r\downarrow 0}\inf\limits_{C}\frac{\HH^s\bigl(A\cap
 B(x,r)\cap C_x\bigr)}{(2r)^s}\geq c
\end{equation}
for $\HH^s$-almost all $x \in A$, where $C_x = \{ x \} + \bigcup C$ and
the infimum is taken over all Borel sets $C\subset G(n,n-m)$ for
which $\gamma_{n,n-m}(C)>\alpha$. Here $\gamma_{n,n-m}$ denotes the unique 
Borel regular probability measure on $G(n,n-m)$ invariant under the
orthogonal group $O(n)$, see \cite[\S 3.9]{ma}. As an immediate
corollary to Mattila's result, under the same assumptions as in
\eqref{thm:salli}, we have
\begin{equation}\label{cor:mattila}
\limsup_{r\downarrow 0}\inf_{V\in G(n,n-m)} \frac{\HH^s\bigl(A\cap
  X(x,r,V,\alpha)\bigr)}{(2r)^s}\geq c  
\end{equation}
for $\HH^s$-almost all $x\in A$, where $c>0$ depends only on $n$, $m$,
$s$, and $\alpha$, see \cite[\S 11]{ma}. 
Although the constant in \eqref{thm:salli} is much better than that of
\eqref{cor:mattila}, still \eqref{cor:mattila} is a significant
improvement of \eqref{thm:salli}: It shows that in the sense of the
measure $\HH^s$, there are arbitrarily small scales such that almost all
points of $A$ are well surrounded by $A$. 

\begin{figure}
\psfrag{x}{$x$}
\psfrag{r}{$r$}
\psfrag{d}{$\delta$}
\begin{center}
\includegraphics[scale=0.8]{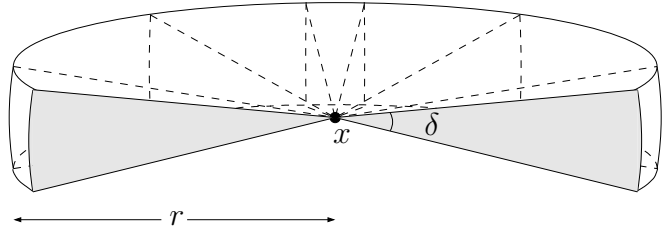}
\end{center}
\caption{The set $X(x,r,V,\alpha)\setminus H(x,\theta,\eta)$ when
  $n=3$, $m=1$, $\alpha=\sin(\delta/2)$, and $\theta$ is pointing
  up from the paper.}
\label{cones}
\end{figure}

In what follows, we shall also allow $m=0$, in which case
$G(n,n-m)=G(n,n)=\{\Rn\}$ and $X(x,r,\Rn,\alpha)=B(x,r)$.
If $\mu$ is a measure on $\Rn$
and $A\subset\Rn$, we use the notation $\mu|_A$ for the restriction
measure, that is $\mu|_A(B)=\mu(A\cap B)$ for $B\subset\Rn$.

The proof of \eqref{thm:mattila} is nontrivial and it is
based on Fubini-type arguments and an elegant use of the so-called
sliced measures. Since the geometry of the cones $X(x,r,V,\alpha)$ is
simpler than that of the cones $C_x$ in \eqref{thm:mattila},
it is natural to ask for an elementary proof of
\eqref{cor:mattila}. In \cite{KS}, such a proof was given and the
technique used there does not require the cones to be symmetric.
Namely, given $s>m$, $0<\alpha<1$, $0<\eta<1$, and
$A\subset\R^n$ with $0<\HH^s(A)<\infty$, it was shown in \cite[Theorem
2.5]{KS} that there is a constant $c>0$ depending only on
$n,m,s,\alpha$, and $\eta$ so that 
\begin{equation}\label{thm:ks}
\limsup_{r\downarrow 0}\inf_{\yli{\theta \in S^{n-1}}{V \in
    G(n,n-m)}} \frac{\HH^s\bigl(A\cap
  X(x,r,V,\alpha)\setminus H(x,\theta,\eta)\bigr)}{(2r)^s}\geq c  
\end{equation}
for $\HH^s$-almost all $x\in A$. Here
$S^{n-1}=\{x\in\R^n\,:\,|x|=1\}$ and 
\begin{equation*}
  H(x,\theta,\eta) = \{ y \in \R^n : (y-x) \cdot \theta > \eta|y-x| \}
\end{equation*}
is the almost half-space centered at $x$ pointing to the direction of
$\theta$ with the opening angle $0<\beta<\pi$ given by
$\cos(\beta/2)=\eta$. 

At first glance, the cones $X(x,r,V,\alpha)\setminus
H(x,\theta,\eta)$ may seem a bit artificial. Let us
look at some special cases. To help the geometrical visualization, it
might be helpful to take $\alpha$ and $\eta$ close to $0$ and
$\theta\in V\cap S^{n-1}$, see Figure \ref{cones}.
When $m=n-1$, the claim \eqref{thm:ks} is equivalent to
\begin{equation}\label{thm:1dim}
\limsup_{r\downarrow 0}\inf_{\varrho \in S^{n-1}}\frac{\HH^s\bigl(A\cap
  X^+(x,r,\varrho,\alpha)\bigr)}{(2r)^s}\geq c(n,s,\alpha)>0,  
\end{equation}
where 
\begin{align*}
X^+(x,r,\varrho,\alpha)&=
\{y\in
B(x,r)\,:\,(y-x)\cdot\varrho>(1-\alpha^2)^{1/2}|y-x|\}\\
&=B(x,r)\cap
H\bigl( x,\varrho,(1-\alpha^2)^{1/2} \bigr).
\end{align*} 
Since $X(x,r,V,\alpha)=X^+(x,r,\varrho,\alpha)\cup
X^+(x,r,-\varrho,\alpha)$ whenever $V=\{t\varrho\,:\,t\in\R\}\in G(n,1)$,
we see from \eqref{thm:1dim} that the cone $X(x,r,V,\alpha)$ in
\eqref{cor:mattila} may be
replaced by $X^+(x,r,\varrho,\alpha)$ when $m=n-1$. This case
was also considered in Mattila \cite{ma2}. 

When $0<m<n-1$, there is no
more natural way to divide the cones $X(x,r,V,\alpha)$ into two or more
similar parts, and we are led to replace the cones
$X^+(x,r,\varrho,\alpha)$ by 
$X(x,r,V,\alpha)\setminus H(x,\theta,\eta)$. However, the main reason
for considering the densities \eqref{thm:ks} in \cite{KS} comes from
porosity. Mattila's result \eqref{thm:1dim} implies
that the lower porosity of the measure $\HH^s|_A$ can not be too
close to the maximum value $\tfrac12$ when $s>n-1$. This leads into a
relatively sharp dimension estimate for lower porous sets with
porosity close to $\tfrac12$, see \cite{ma2} and \cite[\S 11]{ma}. In
a similar manner, the result \eqref{thm:ks} leads to a dimension
estimate for the so called $k$-porous sets, introduced in \cite{KS}. 

When $m=0$, the statement \eqref{thm:ks} is applicable to
all $0<s\leq n$ and reads
\begin{equation}\label{thm:0dim}
\limsup_{r\downarrow 0}\inf_{\theta \in S^{n-1}} \frac{\HH^s\bigl(A\cap
  B(x,r)\setminus H(x,\theta,\eta)\bigr)}{(2r)^s}\geq c(n,s,\eta)>0,
\end{equation}
thus showing that for almost all $x\in A$ the set $A$ (or the measure
$\HH^s|_A$) can not be
concentrated on almost half-balls $B(x,r)\cap H(x,\theta,\eta)$ for
all small scales. Easy
examples, such as $A=S^1\subset\R^2$, show that one can not replace
the almost 
half-spaces $H(x,\theta,\eta)$ by the half-spaces $H(x,\theta,0)$ in
\eqref{thm:0dim}. 

The statement \eqref{thm:ks} as well as its more general formulation
\cite[Theorem 2.6]{KS} deals with measures having finite upper density
with respect to some gauge function. In particular, they do not in
general apply
to packing type measures. Thus there is a need for upper conical
density theorems concerning measures with finite lower density and
(possibly) infinite upper density. In our main result, Theorem
\ref{thm:main}, we 
generalize the result \eqref{thm:ks} for measures 
with finite lower density with respect to an appropriate gauge. The
main application of this 
generalization, Corollary \ref{cor:packing}, is a
conical density theorem for the $s$-dimensional packing measure,
$\PP^s$. Our result may also
be applied to a large collection of Hausdorff and packing type measures
which are determined using a variety of gauges.
Besides the generalizations of \eqref{thm:salli} given in \cite{Su},
there seems to be no conical density theorems of a similar type in the
literature for other than Hausdorff measures.

Theorem \ref{thm:main} may be viewed as a dual result to the known lower
conical density theorems which tell roughly that under certain
conditions, we may find, around typical points, some small half balls
with almost no measure. See, for example, \cite[Theorem 2.1]{Su}.

In \S \ref{sec:poro}, we discuss connections between conical densities
and porosity. Namely, we show how conical density theorems may be used
to obtain 
upper bounds for the porosity of measures. We shall also discuss the
sharpness of our main result using this connection.
Finally, in \S \ref{sec:op} we pose some open problems.

We finish the introduction by setting down some notation. Throughout
the paper, we assume that $h$ is a positive function defined on some
small interval $(0,r_0)$. We shall also assume, for simplicity, that
$h$ is nondecreasing though this is not essential.
If $\mu$ is a Borel measure on $\Rn$ (i.e.\ an outer measure defined on
all subsets of $\Rn$ such that Borel sets are measurable) and
$x\in\Rn$, the upper and lower $\mu$-densities at $x$ 
with respect to $h$ are given by 
\begin{align*}
  \ud{\mu}{x} &= \limsup\limits_{r\downarrow
  0} \frac{\mu\bigl(B(x,r)\bigr)}{h(2r)}, \\
  \ld{\mu}{x} &= \liminf\limits_{r\downarrow
  0} \frac{\mu\bigl(B(x,r)\bigr)}{h(2r)}.
\end{align*} 
If $V\in G(n,m)$, $x\in\Rn$, and $\lambda>0$, we define
\begin{align*}
  V_x(\lambda) &= \{y\in\Rn\,:\,\dist(y-x,V)\leq\lambda\}.
\end{align*}

\section{Conical upper density theorems}

To prove our main result, Theorem \ref{thm:main}, we need the
following two geometrical lemmas. The first one is due to Erd\H{o}s
and F\"uredi \cite{EF}, see also \cite[Lemma 2.1]{KS}.

\begin{lemma} \label{lemma:EF}
  For a given $0<\beta<\pi$, there is $q=q(n,\beta) \in \N$ such that in
  any set of $q$ points in $\R^n$, there are always three points which
  determine an angle between $\beta$ and $\pi$.
\end{lemma}

For $0 < \eta \le 1$ we define $t(\eta) = (\eta^2 +
4)^{1/2}/\eta$ and $\gamma(\eta) = 1/t(\eta)$.
Notice that $t(\eta) \ge 2$ and $\eta/5^{1/2} \le \gamma(\eta) \le
\eta/2$. An easy calculation yields the following, see \cite[Lemma 2.3]{KS}.

\begin{lemma} \label{thm:etamato}
  Suppose $y \in \R^n$, $\theta \in S^{n-1}$, $0<\eta\le 1$,
  $t \ge t(\eta)$, and $\gamma=\gamma(\eta)$. If $z \in \R^n \setminus
  \bigl( B(y,tr) \cup H(y,\theta,\gamma) \bigr)$, then
  $B(z,r) \cap H(y,\theta,\eta) = \emptyset$.
\end{lemma}

Below, we include one more simple lemma.

\begin{lemma}\label{hlemma}
Let $m\geq 0$ be an integer and $h\colon (0,r_0)\to(0,\infty)$. Then
the following conditions are equivalent:
\begin{enumerate}
\item\label{1} There is $r_0>0$ such that
\begin{equation}\label{eq:weakh}
  \frac{h(\eps r)}{\eps^mh(r)}\overset{\eps\downarrow
  0}{\longrightarrow} 0
\end{equation}
uniformly for all $0<r<r_0$.
\item\label{2} There is $s>m$ and $r_0,\varepsilon_0>0$ such that 
  \begin{equation}\label{eq:h}
    h(\varepsilon r)\leq\varepsilon^s h(r)
  \end{equation} 
for all $0<r<r_0$ and $0<\eps<\eps_0$.
\item\label{3} There is $0<c<1$ such that
  \begin{equation*}
    \limsup_{r \downarrow 0} \frac{h(cr)}{h(r)} < c^m.
  \end{equation*}
\end{enumerate}
\end{lemma}

\begin{proof}
By \eqref{1}, there is $0<\delta<1$ and $0<c<1$ such that
$h(\delta r)<c \delta^m h (r)$ for all $0<r<r_0$. Let
$s_0>0$ be such that $\delta^{s_0}=c$ and take
$m<s<m+s_0$ and $0<\varepsilon_0<\delta$ for which
$\varepsilon^{m+s_0}\leq
\delta^{m+s_0}\varepsilon^s$ for all $0<\varepsilon<\varepsilon_0$. Given
$0<\varepsilon<\varepsilon_0$, let $k\in\N$ be such 
that $\delta^{k+1}<\varepsilon\leq\delta^k$. Then
\begin{align*}
h(\varepsilon r)&\leq h(\delta^k r)\leq c^k\delta^{km}
h(r)=\delta^{k(m+s_0)}h(r)
=\varepsilon^{m+s_0}\big(\delta^k/\varepsilon\big)^{m+s_0}h(r)
\\
&\leq\big(\delta^{k+1}/\varepsilon\big)^{m+s_0}\varepsilon^s
h(r)<\varepsilon^s h(r) 
\end{align*}
for all $0<r<r_0$ giving \eqref{2}. That \eqref{3} implies \eqref{1}
follows by a similar reasoning. Finally, notice that \eqref{2}
clearly implies \eqref{3}.
\end{proof}

Next we prove our main result concerning the distribution of measures
with finite lower density.

\begin{theorem}\label{thm:main}
  Let $\alpha,\eta\in (0,1)$ and suppose
  $h\colon(0,r_0)\rightarrow(0,\infty)$ satisfies \eqref{eq:weakh} for
  some $m \in \{0,\ldots,n-1\}$. If $\mu$ is a Borel measure on $\Rn$
  with $\ld{\mu}{x}<\infty$ for $\mu$-almost all $x\in\Rn$ then
  \begin{equation}\label{claim}
    \limsup_{r \downarrow 0} \inf_{\yli{\theta \in S^{n-1}}{V \in
    G(n,n-m)}} \frac{\mu\bigl( X(x,r,V,\alpha)
    \setminus H(x,\theta,\eta)
    \bigr)}{h(2r)} \ge c \ud{\mu}{x}
  \end{equation}
  for $\mu$-almost all $x \in \R^n$. Here $c>0$ is a constant
  depending only on $n,m,\varepsilon_0, s,\alpha$ and $\eta$ where
  $\varepsilon_0>0$ and $s>m$ are as in Lemma \ref{hlemma}.
\end{theorem}

\begin{proof}
  Let us first sketch the main idea of the proof: Suppose our theorem is
  false. Then there is a closed  exceptional set $F\subset\Rn$ with positive
  $\mu$-measure so that for all small scales $r>0$ and for all points
  $x$ of $F$, 
  there are $\theta$ and $V$ so that $\mu\bigl(X(x,r,V,\alpha)\setminus
  H(x,\theta,\eta)\bigr)$ is very small compared to $h(2r)$. A simple covering
  argument on $G(n,n-m)$ implies that at each small ball $B=B(z,r)$ centered
  in $F$, we may fix
  $V\in G(n,n-m)$ so that the measure $\mu\bigl(X(x,r,V,\alpha)\setminus
  H(x,\theta,\eta)\bigr)$ is small for some $\theta$ for a set of points
  $x\in F\cap B$ whose measure is comparable to $h(2r)$. This implies
  that for $\lambda>0$, we may find $y\in F\cap B$ so that the
  measure in $V_y(\lambda r)$ is 
  comparable to $\lambda^m h(2r)$. But our assumption implies that if
  $\lambda$ is small, then this
  measure is essentially contained in at most $q-1$ balls of radius
  $\lambda r$, the number $q$ being determined by Lemma
  \ref{lemma:EF}. Thus, there is a ball $B(w,\lambda r)\subset B$ so
  that $\mu\bigl(F\cap B(w,\lambda r)\bigr)\approx \lambda^m h(2r)$. Iterating
  this, we find a sequence of balls $B_1\supset B_2\supset\cdots$ so
  that $\diam (B_k)\approx\lambda^{k}$ and $\mu(F\cap B_k)\approx
  \lambda^{mk}$. By \eqref{eq:weakh}, this implies $\ld{\mu}{x}=\infty$
  for the point $x$ given by $\{x\}=\bigcap_{k}B_k$. This
  gives a contradiction since we may 
  choose $F$ at the outset so that the lower density $\ld{\mu}{x}$
  is finite for all points of $F$.  

  We shall now verify in detail the steps described heuristically
  above. We assume that $m\geq 1$. 
  The case $m=0$ is easier and is discussed at the end of the proof. 
  We may assume that
  $\mu$ is finite since $\mu$-almost all of $\Rn$ is contained in a
  countable union of open balls, each of finite $\mu$-measure. This
  follows by a straightforward covering argument since
  $\ld{\mu}{x}<\infty$ almost everywhere. Let
  $\varepsilon_0>0$ and $s>m$ be as in Lemma \ref{hlemma}.
  We shall prove that for any finite collection, 
  $\{V^1,\ldots,V^l\}\subset G(n,n-m)$,
  \begin{equation*}
    \limsup_{r \downarrow 0} \inf_{\yli{\theta \in S^{n-1}}{i \in \{ 1,
    \ldots,l \} }} \frac{\mu\bigl( X(x,r,V^i,\alpha)
    \setminus H(x,\theta,\eta) \bigr)}{h(2r)}
    \ge c(n,m,s,\varepsilon_0,\eta,\alpha,l)\ud{\mu}{x}
  \end{equation*}
  for $\mu$-almost all $x \in \R^n$ from which \eqref{claim} follows by
  the compactness of $G(n,n-m)$, see \cite[proof of Theorem 2.5]{KS}
  for details.

  Set $t = \max\{t(\eta), 1+3/\alpha\}$, $\gamma = \gamma(\eta)$, where
  $t(\eta)$ and $\gamma(\eta)$ are as in Lemma \ref{thm:etamato},
  and take $\beta < \pi$
  so that the opening angle of $H(x,\theta,\gamma)$ is smaller than
  $\beta$. Let $q = q(n,\beta)$ be as in Lemma \ref{lemma:EF}.
  Moreover, define $c_1 =2^{m}m^{m/2}$, $c_2=2^n n^{n/2}$,
  $d=\bigl(3 c_1 l (q-1)\bigr)^{-1}$, $\lambda
  =\min\{2^{-1}t^{s/(m-s)}d^{1/(s-m)},\varepsilon_{0}/(3t)\}$, and
  $c=c(n,m,s,\eta,\alpha,l)=\lambda^{n}/(6 c_1 c_2\ell 
  3^s)$. These definitions together with \eqref{eq:h} guarantee the
  following three facts: If $0<r<r_0$, 
  $k\in\N$, $V\in G(n,n-m)$, $z \in \R^n$,
  and $x,y\in V_z(\lambda r)$
  with $|x-y|\geq t\lambda r$, then
  \begin{align}
    B(y,\lambda r) &\subset X(x,V,\alpha),\label{fact1} \\
    h\bigl(6(t\lambda)^k r\bigr) &< 3^s d^k \lambda^{km}
    h(2r), \label{fact2} \\
    d \lambda^{m-s}t^{-s} &\geq 2^{s-m}.\label{fact3}
  \end{align}
  We give some details for the convenience. The claim \eqref{fact1}
  follows since $d(w-x,V)\leq 3\lambda 
  r\leq \alpha(t-1)\lambda r< \alpha |w-x|$ for all $w\in B(y,\lambda r)$ by
  the definition of $t$. To
  prove \eqref{fact2}, we use \eqref{eq:h} to get
  $h\bigl(6(t\lambda)^k r\bigr)\leq 
  3^s t^{ks} \lambda^{ks} h(2r)$. The definition of $\lambda$ easily
  gives $t^{ks}\lambda^{ks}< d^k\lambda^{km}$. Finally, the bound
  \eqref{fact3} comes directly from the definition of $\lambda$. 
  
  Let $0<M<\infty$ and define
  \begin{equation*}
    A=\{x\in\Rn : \ud{\mu}{x}>M\text{ and }\ld{\mu}{x}<\infty\}.
  \end{equation*}
  The set $A$ is Borel since $x\mapsto \ud{\mu}{x}$ and
  $x\mapsto\ld{\mu}{x}$ are Borel functions.
  It suffices to show that
  \begin{align*}
    \limsup_{r \downarrow 0} \inf_{\yli{\theta \in S^{n-1}}{i \in \{ 1,
    \ldots,l \} }} \frac{\mu\bigl( X(x,r,V^i,\alpha)
    \setminus H(x,\theta,\eta) \bigr)}{h(2r)}
    \geq c M
  \end{align*}
  for almost all $x\in A$.
  Suppose to the contrary that there exists
  a set $F \subset A$ with $\mu(F) > 0$ and $0<r_1<r_0$ such that for every
  $x \in F$ and $0 < r < r_1$, there are $i \in \{ 1,\ldots,l \}$ and
  $\theta \in S^{n-1}$ with
  \begin{equation}\label{eq:at}
    \mu\bigl( X(x,r,V^i,\alpha) \setminus
    H(x,\theta,\eta) \bigr) < c M h(2r).
  \end{equation}
  Going into a subset, if necessary, we may assume that
  $F$ is closed.

  Choose $x\in F$ such that $\lim_{r\downarrow 0}
  \mu\bigl(F \cap B(x,r)\bigr)/\mu\bigl(B(x,r)\bigr)=1$ and
  $0<r<r_1/3$ such that
  $\mu\bigl(F \cap B(x,r)\bigr)\geq M h(2r)$. To simplify the notation,
  we assume that $r=1$ and $h(2)=1$. 
  We can do this by replacing $\mu$ by
  $\tilde{\mu}(A)=\mu(rA)/h(2r)$ and $h$ by
  $\tilde{h}(t)=h(rt)/h(2r)$. Our aim is to find $z\in F$ for which
  $\ld{\mu}{z}=\infty$ and this is clearly equivalent to
  $\underline{D}_{\tilde{h}}(\tilde{\mu},z/r)=\infty$.

  Let $B_0=B(x,1)$. Suppose that
  $B_k=B\bigl(x_k, (t\lambda)^k\bigr)$
  has been defined for $k\geq 0$ so that $\mu(F \cap B_k)\geq M
  d^k\lambda^{mk}$. 
  Take $x_{k+1}\in F \cap B_k$ which maximizes the function
  $y \mapsto \mu\bigl(F \cap B(y,(t\lambda)^{k+1})\bigr)$
  in $F \cap B_k$. There is such
  a point because $F \cap B_k$ is compact and the function $y\mapsto
  \mu\bigl(F \cap B(y, (t\lambda)^{k+1})\bigr)$ is upper 
  semicontinuous on $F \cap B_k$.
  Define $B_{k+1}=B\bigl(x_{k+1}, (t\lambda)^{k+1}\bigr)$.
  Our aim is to estimate the measure $\mu(F \cap B_{k+1})$ from
  below. Define, for $i \in \{ 1,\ldots,l \}$,
  \begin{equation*}
  \begin{split}
    \tilde{C_i} = \bigl\{ x \in F \cap B_k : \mu\bigl(
    X&(x,3(t\lambda)^{k},V^i,\alpha) \setminus H(x,\theta,\eta) \bigr) \\ &<
    c M h\bigl(6(t\lambda)^{k}\bigr) \text{ for some }
    \theta \in S^{n-1} \bigr\}.
  \end{split}
  \end{equation*}
  Fix $i \in \{ 1,\ldots,l \}$ for which
  $\mu(\tilde{C_i})\geq\mu(F \cap B_k)/l\geq M d^k \lambda^{mk}/l$
  and take a compact $C_i\subset\tilde{C_i}$ with
  $\mu(C_i)>\mu(\tilde{C_i})/2$.  We may cover the set 
  ${V^i}^\bot \cap B_k$ with
  $c_1\lambda^{-m}$ balls of radius $t^k\lambda^{k+1}$
  and hence there exists $y \in {V^i}^\bot \cap B_k$ for which
  \begin{equation} \label{eq:p_2}
    \mu\bigl(C_i \cap V_y^i(t^k\lambda^{k+1}) \bigr) \ge
    2^{-1}c_{1}^{-1}\ell^{-1}M d^k \lambda^{m(k+1)}.
  \end{equation}

  Next we shall choose $q$ points as follows: Choose a point
  $y_1 \in C_i \cap V_y^i(t^k\lambda^{k+1})$ such that the ball
  $B(y_1,t^k\lambda^{k+1})$ has largest $\mu|_F$ measure among the balls
  centered at $C_i \cap
  V_y^i(t^k\lambda^{k+1})$ with radius $t^k\lambda^{k+1}$. If
  $y_1,\ldots,y_p$, $p \in \{ 1,\ldots,q-1 \}$, have already been
  chosen, we choose $y_{p+1} \in C_i \cap V_y^i(t^k\lambda^{k+1}) \setminus
  \bigcup_{j=1}^p U\bigl(y_j,(t\lambda)^{k+1}\bigr)$ so that the ball
  $B(y_{p+1},t^k\lambda^{k+1})$ has maximal $\mu|_F$ measure among the
  balls centered at $C_i \cap
  V_y^i(t^k\lambda^{k+1}) \setminus 
  \bigcup_{j=1}^p U\bigl(y_j,(t\lambda)^{k+1}\bigr)$ with radius
  $t^k\lambda^{k+1}$. If our process of selecting the points $y_j$
  terminates before the $q$:th step, i.e.\ the balls
  $\bigcup_{j=1}^{p}U\bigl(y_j,(t\lambda)^{k+1}\bigr)$ cover the set
  $F\cap C_i\cap V^{i}_{y}(t^k\lambda^{k+1})$ for some $p<q$, we get
  \begin{equation}\label{eq:processterminates}
  \begin{split}
    \sum_{j=1}^{p}\mu\bigl(F\cap B(y_j,(t\lambda)^{k+1})\bigr) &\geq
    \mu\big(C_i\cap V_y^i(t^k\lambda^{k+1})\big)\\ &\geq 
    2^{-1}c_{1}^{-1}\ell^{-1}M d^k 
    \lambda^{m(k+1)} 
  \end{split}
  \end{equation}
  by \eqref{eq:p_2}.

  Suppose now 
  that the process did not terminate before the $q$:th step. Since 
  the set $V_y^i(t^k\lambda^{k+1}) \cap B_k$ may be covered
  by $c_2\lambda^{m-n}$ balls of radius $t^k\lambda^{k+1}$, using
  \eqref{eq:p_2}, we get
  \begin{equation} \label{eq:p_3}
  \begin{split}
    \mu\bigl( F \cap B(y_q,t^k\lambda^{k+1}) \bigr) \geq c_{2}^{-1}\lambda^{n-m}
    \biggl(&2^{-1} c_{1}^{-1}\ell^{-1}M d^{k}\lambda^{m(k+1)} \\
    &- \sum_{j=1}^{q-1}\mu\bigl( F \cap B(y_j,(t\lambda)^{k+1})
      \bigr) \biggr).
  \end{split}
  \end{equation}
  According to
  Lemma \ref{lemma:EF}, we may choose three points 
  $w,w_1,w_2$ from the set $\{ y_1,\ldots,y_q \}$ such that for each
  $\theta \in S^{n-1}$ there is $j \in \{ 1,2 \}$ for which $w_j \in
  \R^n \setminus \bigl( B(w,(t\lambda)^{k+1}) \cup H(w,\theta,\gamma)
  \bigr)$. We obtain, using Lemma \ref{thm:etamato}, that for each
  $\theta \in S^{n-1}$ there is $j \in \{ 1,2 \}$ such that
  \begin{equation*}
    B(w_j,t^k\lambda^{k+1}) \subset B\bigl(w,3(t\lambda)^{k}\bigr)
    \setminus H(w,\theta,\eta)
  \end{equation*}
  and hence \eqref{fact1} implies that also
  \begin{equation}\label{inclusion}
    B(w_j,t^k\lambda^{k+1}) \subset X\bigl(w,3(t\lambda)^k,V^i,\alpha\bigr)
    \setminus H(w,\theta,\eta),
  \end{equation}
  see Figure \ref{balls}.
  \begin{figure}
    \psfrag{V}{$V_{y}^{i}(t^k\lambda^{k+1})$}
    \psfrag{X}{$X(w,3(t\lambda)^k,V^i,\alpha)$}
    \psfrag{a}{$w_1$}
    \psfrag{b}{$w_2$}
    \psfrag{w}{$w$}
    \psfrag{g}{$\delta$}
    \psfrag{B}{$B_k$}
    \begin{center}
    \includegraphics[scale=0.8]{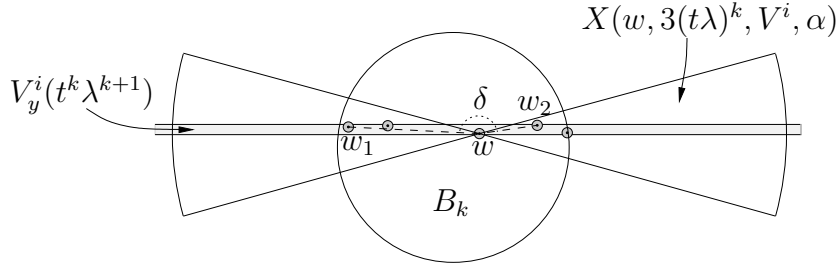}
    \end{center}
    \caption{Illustration for the proof of Theorem \ref{thm:main}. The
    angle $\delta$ formed by the points $w_1,w$, and $w_2$ is greater
    than $\beta$.} 
    \label{balls}
  \end{figure}
  Since $w \in C_i$ there is $\theta\in S^{n-1}$ so that
    $\mu\bigl( X(w,3(t\lambda)^k,V^i,\alpha) \setminus H(w,\theta,\eta) \bigr)
     < c M h\bigl(6(t\lambda)^{k}\bigr)$. Choosing $j\in\{1,2\}$ for
    which \eqref{inclusion} holds, we get
  \begin{align}
    \mu\bigl(F\cap B(y_q,t^k\lambda^{k+1})\bigr)&\leq\mu\bigl(F\cap
    B(w_j,t^k\lambda^{k+1})\bigr)\notag\\
    &\leq \mu\bigl( X(w,3(t\lambda)^k,V^i,\alpha)\label{estimate}
    \setminus H(w,\theta,\eta) \bigr)\\
    & < c M h\bigl(6(t\lambda)^{k}\bigr).\notag
  \end{align}
  Consequently, using 
  \eqref{eq:p_3}, \eqref{estimate}, \eqref{fact2}, and the definitions
  of $c$, $c_1$, $c_2$, and $d$, we get
  \begin{align*}
    \sum_{j=1}^{q-1} \mu\bigl( F \cap B(y_j,(t\lambda)^{k+1}) \bigr) &> 
    2^{-1}c_{1}^{-1}\ell^{-1}M d^{k}\lambda^{m(k+1)}- 
    c_2 c M h\bigl(6(t\lambda)^k\bigr)\lambda^{m-n}\\
    &> 2^{-1} c_{1}^{-1}\ell^{-1}M d^{k}\lambda^{m(k+1)}-c_2 c M
    3^s d^k 
    \lambda^{m(k+1)}\lambda^{-n}\\
    &=  3^{-1}c_{1}^{-1}\ell^{-1}M d^{k}\lambda^{m(k+1)}\\
    &= (q-1)M d^{k+1}\lambda^{m(k+1)}.
  \end{align*}
  It follows that there is $y_j \in \{ y_1,\ldots,y_{q-1} \}$ for
  which $\mu\bigl( F \cap B(y_j,(t\lambda)^{k+1}) \bigr) \geq
  M (d\lambda^m)^{k+1}$. Inspecting the above calculation, we see that
  this is true also if
  \eqref{eq:processterminates} holds. Thus we get
  \begin{equation}\label{eq:muk}
    \mu(F \cap B_{k+1})\geq M (d\lambda^m)^{k+1} 
  \end{equation}
  and this remains true for all $k\in\mathbb{N}$. 

  Let $z=\lim_{k\rightarrow\infty} x_k$. Since $t\lambda \le 1/3$, we have
  $|z-x_k|\leq\sum_{i=k}^{\infty}(t\lambda)^i<2(t\lambda)^k$. Thus
  $B_k\subset B\bigl(z,3(t\lambda)^k\bigl)$ for all $k\in\N$.
  If $(t\lambda)^{k+1}\leq
  r'<(t\lambda)^{k}$, then $3r'<(t\lambda)^{k-1}$, and hence,
  using  \eqref{eq:muk}, \eqref{eq:h}, and \eqref{fact3}, we get 
  \begin{align*}
    \frac{\mu\bigl(B(z,3r')\bigr)}{h(6r')}
    &\geq \frac{\mu(B_{k+1})} {
    h\bigl(2(t\lambda)^{k-1}\bigr)} 
    > \frac{M d^{k+1}\lambda^{m(k+1)}}{h\bigl(2(t\lambda)^{k-1}\bigr)} \\
    &= M d^2 \lambda^{2m}\bigl(d \lambda^{m-s}t^{-s}\bigr)^{k-1}
    \frac{(t\lambda)^{s(k-1)}}{h\bigl(2(t\lambda)^{k-1}\bigr)} \\
    &\geq \frac{M d^2 \lambda^{2m} 2^{(s-m)(k-1)}}{h(2)}\longrightarrow\infty 
  \end{align*}
  as $r'\downarrow 0$. This implies $\ld{\mu}{z}=\infty$,
  giving a contradiction since $z\in F$. 
  This completes the proof in the case $m\geq 1$.

  When $m=0$, the proof is actually easier since we do not need to consider
  the slices $V_{i}^{y}$. We argue by contradiction that there is a compact set
  $F$ with $\mu(F)>0$ so that $\ld{\mu}{x}<M$ and
  \eqref{eq:at} is satisfied for all $x\in F$ (the cones $X(x,r,V^i,\alpha)$ 
  are replaced by $B(x,r)$, $l=1$, and the infimum is only over all $\theta\in
  S^{n-1}$). Then we define $B_0$ such that $\mu(F\cap B_0)\geq M
  h\bigl(\diam(B_0)\bigr)$ and for $k\geq 0$ we choose the balls
  $B\bigl(y_j,(t\lambda)^{k+1}\diam(B_0)/2\bigr)$ for $y_1,\ldots,
  y_q\in F\cap B_k$ as 
  above. Finally, we use Lemma \ref{lemma:EF} to get a lower bound for
  $\mu(F\cap B_{k+1})$ yielding a point $z\in F$ for which
  $\ld{\mu}{z}=\infty$. 
\end{proof}

Let us now consider the most important special cases of Theorem
\ref{thm:main}. Let $h_s(r) =
r^s$ as $r \ge 0$. As noted in the
introduction, Theorem \ref{thm:main} is a generalization of
\eqref{thm:ks}. This follows from the well known fact that 
\[2^{-s} \le \uds{\HH^s|_A}{x} \le 1\]
for $\HH^s$-almost all $x\in A$ provided that $A\subset\Rn$ with
$0<\HH^s(A)<\infty$. 
The most important improvement in Theorem \ref{thm:main} compared to
\eqref{thm:ks} is related to the $s$-dimensional packing measure, $\PP^s$. See
\cite[\S 5.10]{ma} for the definition. If $A\subset\Rn$ with
$0<\PP^s(A)<\infty$ then
\[\lds{\PP^s|_A}{x}=1\]
for $\PP^s$-almost all $x\in A$, see \cite[Theorem 6.10]{ma}. Thus we
get the following corollary: 

\begin{corollary} \label{cor:packing}
  Suppose $0 \le m < s \le n$ and $0 < \alpha,\eta \le 1$. Then there
  is a constant $c=c(n,m,s,\alpha,\eta)>0$ such that
  \begin{align}
    \limsup_{r \downarrow 0} \inf_{\yli{\theta \in S^{n-1}}{V \in
    G(n,n-m)}} &\frac{\PP^s\bigl( A \cap X(x,r,V,\alpha)
    \setminus H(x,\theta,\eta) \bigr)}{(2r)^s}\label{ud}\\
    &\geq c\,\uds{\PP^s|_A}{x} \geq c \notag
  \end{align}
  for $\PP^s$-almost every $x \in A$ whenever $A\subset\R^n$ with
  $0<\PP^s(A)<\infty$.
\end{corollary}

It is remarkable to note that the upper
density $\uds{\PP^s|_A}{x}$ may be infinity
almost everywhere on the set $A$. In this case Corollary
\ref{cor:packing} states that also the upper density \eqref{ud} is
infinity for $\PP^s$-almost every $x\in A$.

For many fractals some other gauge function than $h_s$ might be more useful
in measuring the fractal set in a delicate manner.
  Denote the
  Hausdorff and packing measures constructed using the gauge $h$ by
  $\HH_h$ and $\PP_h$, respectively. See 
  \cite[\S 4.9]{ma} and \cite[Definition 3.2]{Cu} for the definitions.
  If
  $A,B\subset\Rn$, $0<\HH_h(A)<\infty$, $0<\PP_h(B)<\infty$,
  $\mu=\HH_h|_A$, and $\nu=\PP_h|_B$, then
  $\liminf_{r\downarrow0}h(r)/h(2r)\leq\ud{\mu}{x}\leq 1$ for $\mu$-almost
  every $x\in\Rn$ and
  $\ld{\nu}{x}=1$ for $\nu$-almost
  every $x\in\Rn$. Thus Theorem \ref{thm:main} may be applied to
  measures $\mu$ and $\nu$ provided that $h$ satisfies any of the
  conditions \eqref{1}--\eqref{3} of Lemma \ref{hlemma}.
  These conditions hold for functions such as
  $h(r)=r^s/\log(1/r)$ or 
  $h(r)=r^s\log(1/r)$, $s>m$. However, some gauge functions such as
  $h(r)=r^m/\log(1/r)$ fail to satisfy them although
  $\lim_{r\downarrow 0}h(r)/r^m=0$. For this
  gauge, Theorem \ref{thm:main} is not even true as will be shown in
  Proposition \ref{pro:por}.

\section{Porosity and conical densities}\label{sec:poro}

In this section we discuss relations between conical upper density
theorems and porosity of measures.
Our application concerns the following definition of lower porosity of
measures. 
Let $k$ and $n$ be integers with $1\le k\le n$. For all locally finite
Borel measures $\mu$ in $\mathbb R^n$, $x\in\mathbb R^n$, $r>0$, and
$\varepsilon>0$,
we set
\begin{align*}
  \por_k(\mu,x,r,\varepsilon)=\sup\{\varrho : \;&\text{there are distinct }
    z_1,\ldots,z_k\in\mathbb R^n\setminus\{x\} \text{ such that }\\
 &B(z_i,\varrho r)\subset B(x,r)\text{ and }\mu\bigl(B(z_i,\varrho r)\bigr)
   \le\varepsilon\mu\bigl(B(x,r)\bigr)\\
 &\text{for every }i\text{ and }(z_i-x)\cdot(z_j-x)=0\text{ if }j\neq i\}. 
\end{align*}
The \emph{$k$-porosity} of $\mu$ at a point $x$ is defined to be
\begin{equation*}
  \por_k(\mu,x)=\lim_{\varepsilon\downarrow 0}\liminf_{r\downarrow 0}
   \por_k(\mu,x,r,\varepsilon),
\end{equation*}
When $k=1$, our definition of $\por_1$ agrees with the lower porosity
of measures introduced by Eckmann, J\"arvenp\"a\"a and J\"arvenp\"a\"a in
\cite{EJJ}. When $k>1$, our definition of $k$-porosity is a natural
generalization of the $k$-porosity of sets studied in \cite{JJKS} and
\cite{KS}. For a
motivation, examples, and more information 
on dimension of lower porous sets and measures, consult \cite{JJ} and
\cite{KS}. It is possible that $\por_k(\mu,x)>1/2$ in a single point
but $\por_k(\mu,x)\leq 1/2$ for almost every $x$ for any Borel measure
$\mu$, see \cite[p.\ 4]{EJJ}.

If $0<\alpha<1$ and $m,n\in\mathbb{N}$, we denote
\begin{align*}
V&=\{x\in\Rn\,:\,x_i=0 \text{ for all }i=1,\ldots, n-m\},\\
C&=\{x\in\Rn\,:\,x_i>0\text{ for }i=1,\ldots,n-m\},
\end{align*}
 and
$\theta=(n-m)^{-1/2}\sum_{i=1}^{n-m}e_i\in S^{n-1}$ and define 
\[\eta(\alpha,m,n)=\sup\{\eta\geq 0\,:\,C\cap X(0,V,\alpha)\subset H(0,\theta,\eta)\},\] 
where
$X(x,V,\alpha)=X(x,V,\infty,\alpha)=\{y\in\Rn\,:\,\dist(y-x,V)<\alpha|y-x|\}$.
Moreover, if $0<\eta<\eta(\alpha,n,m)$, we put
$x_0=\sum_{i=1}^{n-m}e_i\in\Rn$ and
\begin{equation}\label{cmato}
\tilde{c}(\eta)=\tilde{c}(\eta,\alpha,n,m)=\inf\{r>0\,:\,C\cap
X(0,V,\alpha)\setminus B(0,r)\subset H(x_0,\theta,\eta)\}.
\end{equation} 
By simple geometric inspections, one checks that $\eta>0$ and
$\tilde{c}<\infty$ though the exact values may be hard to compute.

\begin{theorem}\label{thm:poro}
Let $h$ satisfy the doubling condition
\begin{equation}\label{doubling}
  \limsup_{r\downarrow 0}h(2r)/h(r)<\infty
\end{equation}
and suppose further that
\begin{equation}\label{hunif}
  h(\varepsilon r)/h(r)\overset{\varepsilon\downarrow
    0}{\longrightarrow} 0
\end{equation}
uniformly for all $0<r<r_0$.
Assume that $0\leq m < n$, $0<\alpha<1$, and
$0<\eta<\eta(\alpha,n,m)$. Let $\mu$ be a Borel measure on
$\mathbb{R}^n$ with $0<\ud{\mu}{x}<\infty$ for
$\mu$-almost all $x \in \R^n$ and suppose there is $c>0$ such that 
\begin{equation}\label{cdensity}
  \limsup_{r \downarrow 0} \inf_{\yli{\theta \in S^{n-1}}{V \in
    G(n,n-m)}} \frac{\mu\bigl( X(x,r,V,\alpha)
    \setminus H(x,\theta,\eta)
    \bigr)}{h(2r)} \ge c \ud{\mu}{x}
\end{equation}
for $\mu$-almost every $x \in \R^n$.
Then $\por_{n-m}(\mu,x)\leq 1/2-c'$ for $\mu$-almost every $x$, where
$c'>0$ is a constant depending only on $n,m,\alpha,\eta,c$, and $h$.
\end{theorem}

\begin{proof}
The argument is purely geometric though a bit technical. The idea is similar
to those in the proofs of \cite[Theorem 3.2]{KS} and \cite[Theorem
11.14]{ma}. 

Denote $k = n-m$ and
suppose that $\por_k(\mu,x)>\varrho>\sqrt{2}-1$ in a measurable set
$A\subset\Rn$ with $\mu(A)>0$. Let 
$t=(1-2\varrho)^{-1/2}$ and
$\delta=t\bigl(1-\varrho-(\varrho^2+2\varrho-1)^{1/2}\bigr)$. Then  
\begin{equation}\label{Hinc}
H(x+\delta r\theta,\theta)\cap B(x,r)\subset B(z,\varrho tr)
\end{equation}
whenever $\theta\in S^{n-1}$ and $B(z,\varrho tr)\subset B(x,tr)$, see
\cite[Lemma 3.1]{KS}. Here $H(x,\theta) = H(x,\theta,0)$. Since
$\delta=\delta(\varrho)\downarrow 0$ as 
$\varrho\uparrow 1/2$, it suffices to find a positive lower bound
for $\delta$ depending only on $c$, $h$, $\alpha$, $\eta$, $n$, and $m$.

By \eqref{cdensity}, we may find $x\in A$ for 
which $0<\ud{\mu}{x}=M<\infty$ and 
\begin{equation*}
  \limsup_{r \downarrow 0} \inf_{\yli{\theta \in S^{n-1}}{V \in
  G(n,k)}} \frac{\mu\bigl(X(x,r,V,\alpha) \setminus
  H(x,\theta,\eta) \bigr)}{h(2r)} \ge cM.
\end{equation*}  
Using \eqref{doubling}, we may choose $\varepsilon>0$
so small that 
\begin{equation}\label{eq:epsilon}
\varepsilon  h(2tr)< h(2\tilde{c}\delta r) 
\end{equation}
for all $0<r<r_0$, where $\tilde{c}=\tilde{c}(\eta)$ is as in \eqref{cmato}. Next choose
$0<r_1<r_0$ such that
\begin{equation}\label{localporo}
\por_k(\mu,x,r,\varepsilon/k)>\varrho\quad \text{and}
\quad\mu\bigl(B(x,r)\bigr)<2M h(2r)
\end{equation} 
for all $0<r<r_1$. Now we take
$0<r<\min\{r_1/t,r_1/(2\tilde{c}\delta)\}$ such that 
\begin{equation}\label{conemeas}
  \inf_{\yli{\theta \in S^{n-1}}{V \in
  G(n,k)}} \mu\bigl(X(x,r,V,\alpha) \setminus
  H(x,\theta,\eta) \bigr)>c M h(2r)/2.
\end{equation}
Using \eqref{localporo}, we find $z_1,\ldots,z_k\in
B\bigl(x,(1-\roo)r\bigr) \setminus \{x\}$ with
$(z_i-x)\cdot(z_j-x)=0$ as $i\neq j$ and $\mu\bigl(B(z_i,\varrho
tr)\bigr)\leq\varepsilon\mu\bigl(B(x,tr)\bigr)/k$ for all
$i \in \{1,\ldots,k\}$. In particular, 
\begin{equation}\label{smallmeas}
\mu\biggl(\bigcup_{i=1}^{k}B(z_i,\varrho
tr)\biggr)\leq\varepsilon\mu\bigl(B(x,tr)\bigr)\leq2\varepsilon M h(2tr).
\end{equation} 
Let $\theta_i=(z_i-x)/|z_i-x|$ for $i \in \{1,\ldots,k\}$. Applying
\eqref{Hinc}, we see that 
$H(x+\delta r\theta_i,\theta_i)\cap B(x,r)\subset B(z_i,\varrho tr)$
for every $i$. 
If $V\in G(n,k)$ is the $k$-plane spanned by the
vectors $\theta_1,\ldots,\theta_k$ and
$\theta=-k^{1/2}\sum_{i=1}^{k}\theta_i$ then we
conclude that  
\begin{gather*}
\bigl(X(x,r,V,\alpha)\setminus
H(x,\theta,\eta)\bigr)\setminus\bigcup_{i=1}^{k} B(z_i,\varrho
tr)\\
\subset\bigl(X(x,r,V,\alpha)\setminus
H(x,\theta,\eta)\bigr)\setminus\bigcup_{i=1}^{k} H(x+\delta
r\theta_i,\theta_i)
\subset B(x,\tilde{c}\delta r)
\end{gather*} 
using the definition of $\tilde{c}$ for the last inclusion.

Using \eqref{conemeas}, the above inclusion, the latter condition of
\eqref{localporo}, \eqref{smallmeas}, and \eqref{eq:epsilon}, we 
conclude that 
\begin{align*}
  c M h(2r)/2&<\mu\bigl(X(x,r,V,\alpha) \setminus
  H(x,\theta,\eta) \bigr)\\
  &\leq 2 M h(2\tilde{c}\delta r) + 2\varepsilon
  Mh(2tr) \leq 4 M h(2\tilde{c}\delta r).
\end{align*}
This reduces to $h(2\tilde{c}\delta r)/h(2r)> c/8$ and thus by
\eqref{hunif}, we must have $\delta>\delta_0$ for $\delta_0>0$
depending only on $c$, $h$, $n$, $\alpha$, and $\eta$.
\end{proof}

As an immediate consequence of Theorems \ref{thm:main} and
\ref{thm:poro}, we get the following corollary for the $k$-porosity of
Hausdorff type measures:

\begin{corollary}\label{cor:poro}
Suppose $h$ and $\mu$ satisfy the assumptions of Theorem
\ref{thm:main}, \eqref{doubling}, and $0<\ud{\mu}{x}<\infty$ almost
everywhere. Then $\por_{n-m}(\mu,x)<1/2-c$, where
$c>0$ is a constant depending only on $m$, $n$, $s$,
and $\varepsilon_0$. 
\end{corollary}

When $m=n-1$ and $h=h_s$, Corollary \ref{cor:poro} is a special case of
\cite[Corollary 2.9]{JJ}. 

We do not know if it is possible to find weaker conditions for $h$ than
the ones in Lemma \ref{hlemma} under which Theorem \ref{thm:main}
holds. However, we
may use Theorem \ref{thm:poro} to rule out some possible generalizations.

\begin{proposition}\label{pro:por}
Suppose $h$ satisfies
\eqref{doubling} and \eqref{hunif}. Suppose further that there is an
integer $1\leq m\leq n-1$ and a decreasing
sequence $(r_j)$ for which $h(r_{j+1})\geq 2^{m-n}
(r_{j+1}/r_{j})^m h(r_j)$ and $r_j/r_{j+1}\to\infty$ as
$j\rightarrow\infty$. Then there is a measure $\mu$ on $\Rn$ for which 
  $0<\ud{\mu}{x}<\infty$ for $\mu$-almost all $x \in \R^n$ and 
\begin{equation}\label{cdensity0}
\limsup_{r \downarrow 0} \inf_{\yli{\theta \in S^{n-1}}{V\in G(n,n-m)}} \frac{\mu\bigl(
   X(x,r,V,\alpha) 
    \setminus H(x,\theta,\eta)
    \bigr)}{h(2r)} = 0
\end{equation}
for $\mu$-almost every $x \in \R^n$ and for all $0<\alpha<1$ and
$0<\eta<\eta(\alpha)$. Here $\eta(\alpha)=\eta(\alpha,m,n)$ is as in
Theorem \ref{thm:poro}.
\end{proposition}

\begin{proof}
We may assume that $r_{j+1}<r_j/2$ for all $j$. Let
$\tilde{h}(r)=r^{-m}h(r)$. Then 
$\tilde{h}(r_{j+1})\geq 2^{m-n}\tilde{h}(r_j)$ for all $j\in\mathbb{N}$. Let
$Q\subset\R^{n-m}$ be a closed cube with side-length $r_0$ and let
$Q_{1,1},\ldots Q_{1,2^{n-m}}\subset I$ 
be the closed cubes located at the corners of $Q$ with side-length
$r_1$. In a similar manner, 
divide $Q_{1,1}, Q_{1,2^{n-m}}$ into totally $2^{2(n-m)}$ subcubes of
side-length $r_2$, say $Q_{2,1},\ldots, Q_{2,2^{2(n-m)}}$. Continuing
in this manner, we define a Cantor type set
$C=\bigcap_{j\in\mathbb{N}}\bigcup_{i=1}^{2^{j(n-m)}}Q_{j,i}\subset\R^{n-m}$.
Since arbitrary covers $\{E_k\}_k$ of $C$ are
reduced to finite covers of the sets $Q_{j,i}$, so that
$\sum_{k}\tilde{h}\bigl(\diam(E_k)\bigr)\geq c
\sum_{i}\tilde{h}\bigl(\diam(Q_{j,i})\bigr)$ 
for a constant $c=c(n,m)>0$, we easily obtain $\mathcal{H}_{\tilde{h}}(C)>0$.
If $A=C\times[0,1]^m$ then, by applying the
calculations done in \cite[Theorem 7.7]{ma}, we have
$\HH_h(A)>0$.
Now we may find a compact $F\subset A$ with
$0<\HH_h(F)<\infty$, see \cite{Ho}. For $\mu=\HH_h|_F$, we then have
$0<\ud{\mu}{x}<\infty$ for 
$\mu$-almost all $x\in \Rn$. Since
$r_j/r_{j+1}\rightarrow \infty$, it is easy to see that
$\por_{n-m}(\mu,x)=1/2$ for $\mu$-almost every $x\in F$. By Theorem
\ref{thm:poro}, this implies 
\eqref{cdensity0} for $\mu$-almost all $x$ whenever $0<\alpha<1$ and
$0<\eta<\eta(\alpha)$.
\end{proof}

\begin{remark}\label{rem:fixed}
Inspecting the proofs of Proposition \ref{pro:por} and Theorem
\ref{thm:poro}, it is easily seen that $V$ may be fixed in 
\eqref{cdensity0}.  
\end{remark}

Let us compare the assumptions of the above proposition with 
the assumptions of Theorem \ref{thm:main}. Recall, by Lemma
\ref{hlemma}, that in Theorem \ref{thm:main} our assumption for $h$
is: There is $0<c<1$ such that
$\limsup_{r\downarrow 0} h(cr)/h(r)<c^{m}$. On the other hand, if
\begin{equation*}
\liminf_{r\downarrow 0} h(cr)/h(r)\geq c^{m}
\end{equation*}
for all $0<c<1$ then
the assumptions of Proposition \ref{pro:por} are clearly
satisfied. This shows that Theorem \ref{thm:main} does not hold for gauge
functions such as $h(r)=r^m/\log(1/r)$ when $m>0$. 

\section{Open problems}\label{sec:op}

We discuss below some of the questions raised by Theorem \ref{thm:main}.

\begin{question}
Most measures are so unevenly distributed that
there are no 
functions that could be used to approximate the measure in small
balls. For these measures it is natural to study upper densities such
as 
\[\limsup_{r\downarrow 0}\frac{\mu\bigl(X(x,r,V,\alpha)\bigr)}{\mu\bigl(B(x,r)\bigr)}\,.\]
In order to bound these densities from below, we need to guarantee that
the measure $\mu$ is not concentrated in too small regions. One way to
do this is to impose bounds on the dimension of the measure.
We pose the following open problem. It is stated
here in its simplest form though natural generalizations arise by
analogy with 
\eqref{thm:salli}--\eqref{thm:ks}:
Suppose that $\mu$ is a Borel measure on $\Rn$ whose packing
dimension, $\dimp(\mu)$, equals $s$ (see
\cite[\S 10]{Fa}). If $0<\alpha<1$, $m\in\N$ with $m<s$, and $V\in G(n,n-m)$,
is it true that
\begin{equation}\label{q1}
\limsup_{r\downarrow
  0}\frac{\mu\bigl(X(x,r,V,\alpha)\bigr)}{\mu\bigl(B(x,r)\bigr)}\geq c
\end{equation}
for $\mu$-almost every $x\in\Rn$, where $c>0$ depends only
on $n,m,s$, and $\alpha$? If $\mu$ satisfies almost everywhere the doubling condition
\begin{equation}\label{mudoubling}
\limsup_{r\downarrow
  0}\frac{\mu\bigl(B(x,2r)\bigr)}{\mu\bigl(B(x,r)\bigr)}<d<\infty,
\end{equation}
the answer is known to be yes: In fact, for any Borel measure $\mu$
that satisfies \eqref{mudoubling} and is purely $m$-unrectifiable in
the sense that $\mu(E)=0$ for all $m$-rectifiable sets $E\subset\Rn$,
the claim \eqref{q1} holds with a constant $c=c(d,\alpha)>0$. This
follows by inspecting the proof of \cite[Lemma 15.14]{ma}.
\end{question}

\begin{question}
A related question concerning purely unrectifiable sets is the
following: Suppose that $E\subset\Rn$ is purely $m$-unrectifiable,
$0<\mathcal{H}^m(E)<\infty$, and $\mu=\mathcal{H}^m|_E$. Is 
\begin{equation*}
\limsup_{r\downarrow
  0}\inf_{V\in G(n,n-m)}\frac{\mu\bigl(X(x,r,V,\alpha)\bigr)}{(2r)^m}\geq c(n,m,\alpha)>0
\end{equation*}
for $\mu$-almost every $x$? This would be the analogy of
\eqref{cor:mattila} for purely unrectifiable sets. The analogy of
\eqref{thm:salli} in this case is well known. On the other hand, the
analogy of
\eqref{thm:ks} does not hold under these assumptions, even if we fix
$V$. A set of Besicovitch \cite[p.\ 327]{Bes} serves as a counterexample. 
\end{question}

\begin{question}
Inspecting Proposition \ref{pro:por}, one recognizes
that there are no gauge functions satisfying its assumptions when
$m=0$. This leads to ask if Theorem \ref{thm:main} for $m=0$ is
true for all gauge functions. That is, whether for all $0<\eta<1$
there is $c=c(n,\eta)>0$ such that
 \begin{equation*}
    \limsup_{r \downarrow 0} \inf_{\theta \in S^{n-1}}
    \frac{\mu\bigl(B(x,r)
    \setminus H(x,\theta,\eta)
    \bigr)}{h(2r)} \ge c \ud{\mu}{x} 
\end{equation*} 
for all gauge functions $h$, all Borel measures $\mu$, and $\mu$-almost
every $x \in \R^n$?
When $n=1$, this is known and reads
\[\limsup_{r\downarrow 0}\frac{\min\{\mu([x,x+r]),\mu([x-r,x])\}}{h(2r)}\geq
  \ud{\mu}{x}/4\] 
for $\mu$-almost all $x\in\R$. This follows from the proof of \cite[Theorem
  3.1]{Su2}. 
\end{question}

\begin{question}
When $k>1$, we do not know if Theorem \ref{thm:poro} holds for
packing type measures, that is, for measures with
$0<\ld{\mu}{x}<\infty$. When $k=1$, a more general result is obtained in
a forthcoming paper \cite{bjjkrss}. 
\end{question}

\bibliographystyle{abbrv}
\bibliography{conical.bib}

\end{document}